\documentclass{amsart}

\usepackage[utf8]{inputenc}

\usepackage{amssymb}

\usepackage{color}

\usepackage{hyperref}

\usepackage{amsrefs}

\BibSpec{collection.article}{%
    +{}  {\PrintAuthors}                {author}
    +{,} { \textit}                     {title}
    +{.} { }                            {part}
    +{:} { \textit}                     {subtitle}
    +{,} { \PrintContributions}         {contribution}
    +{,} { \PrintConference}            {conference}
    +{}  {\PrintBook}                   {book}
    +{,} { }                            {booktitle}
    +{,} { }                            {publisher}
    +{,} { \PrintDateB}                 {date}
    +{,} { pp.~}                        {pages}
    +{,} { }                            {status}
    +{,} { \PrintDOI}                   {doi}
    +{,} { available at \eprint}        {eprint}
    +{}  { \parenthesize}               {language}
    +{}  { \PrintTranslation}           {translation}
    +{;} { \PrintReprint}               {reprint}
    +{.} { }                            {note}
    +{.} {}                             {transition}
    +{}  {\SentenceSpace \PrintReviews} {review}
}

\usepackage{graphicx}

\usepackage{xcolor}
\definecolor{Blue}{HTML}{3498DB}

\usepackage[shortlabels]{enumitem}

\usepackage{pgfplots}
\pgfplotsset{compat=newest}

\newtheorem{thm}{Theorem}[section]

\newtheorem{lem}[thm]{Lemma}
\newtheorem{cor}[thm]{Corollary}

\newtheorem{thmalpha}{Theorem}

\newtheorem{lemalpha}[thmalpha]{Lemma}

\newtheorem{conj}[thm]{Conjecture}

\theoremstyle{definition}

\theoremstyle{remark}

\newtheorem{remark}[thm]{Remark}

\numberwithin{equation}{section}

\newcommand{\C}{\mathbb{C}} 
\newcommand{\R}{\mathbb{R}} 
\newcommand{\N}{\mathbb{N}} 
\newcommand{\D}{\mathbb{D}}
\newcommand{\T}{\mathbb{T}}
\newcommand{\Hol}{\mathcal{H} \left( \mathbb{D} \right)}

\begin{document}


	\title[On a conjecture about contractive inequalities]{On a conjecture about contractive inequalities for weighted Bergman spaces}


	\author{Adrián Llinares}
	\address{Departamento de Matemáticas, Universidad Autónoma de Madrid, 28049 Madrid, Spain}
	\email{adrian.llinares@uam.es}
	
	\date{\today}
	\subjclass[2020]{30H10, 30H20, 47A30, 47B91}
	\keywords{Extremal problems, norm of inclusions, Hardy spaces, weighted Bergman spaces, contractive inclusions}



	\begin{abstract}
		We are interested in the norm of the inclusion between the standard weighted Bergman spaces $A^2_\alpha$ and $A^p_{\frac{p}{2} (\alpha + 2)-2}$, $p > 2$, which is conjectured to be contractive by O.F. Brevig, J. Ortega-Cerdà, K. Seip and J. Zhao (in the limit case $\alpha = -1$) and E. Lieb and J.P. Solovej (for $\alpha > -1$). We provide some partial answers and improvements of some recent results.
	\end{abstract}


	\maketitle


	
	\section{Introduction}
	
		\subsection{Notation and definitions}
		
			As is usual, let $\D$ and $\T$ denote the open unit disk of $\C$ and its boundary, respectively. Let $\Hol$ be the set of all holomorphic functions in $\D$. For $p > 0$, we will write the integral $p$-means of $f \in \Hol$ as
			\[
				M_p(r, f) := \left( \dfrac{1}{2\pi} \int_0^{2\pi}  \left | f \left( r e^{it} \right) \right |^p \: dt \right)^\frac{1}{p}, \quad r \in (0, 1), \\
			\]
			and the Hardy space $H^p$ is defined as
			\[
				H^p := \left \{ f \in \Hol \: : \: \| f \|_{H^p} := \sup_{r \in (0,1)} \{ M_p(r,f) \} < \infty\right \}.
			\]
			
			Let $dA(z):= \frac{1}{\pi} \: dx \, dy$ be the normalized Lebesgue area measure of $\D$. For $p > 0$ and $\alpha > -1$, we define the standard weighted Bergman space $A^p_\alpha$ as the set of all analytic functions $f$ such that
			\[
				\| f \|_{A^p_\alpha} := \left( \left( \alpha + 1 \right ) \int_\D (1 - |z|^2)^\alpha |f(z)|^p \: dA(z) \right)^\frac{1}{p} < \infty.
			\]
			
			If $\alpha = 0$, we will simply write $A^p$ instead of $A^p_0$.
			
			It is widely known that Hardy spaces are the limit of weighted Bergman spaces in the sense that
				\begin{equation} \label{LimApalpha}
					\lim_{\alpha \rightarrow -1^+} \| f \|_{A^p_\alpha} = \| f \|_{H^p},
				\end{equation}
				and consequently from now on we will understand the space $A^p_{-1}$ as $H^p$ (for example, see \cite{MR2076585}).
			
			Point evaluations are bounded functionals in $A^p_\alpha$, and the following sharp inequality
				\begin{equation} \label{AcotacionApalpha}
					| f (z) | (1 - |z|^2)^\frac{\alpha + 2}{p} \leq \| f \|_{A^p_\alpha}, \quad \forall z \in \D,
				\end{equation}
				holds (see \cite{MR1120512}, for example). Moreover, we know that, for a fixed $f$, the left-hand side of \eqref{AcotacionApalpha} tends uniformly to 0 when $z$ tends to the boundary.
				
			For every $\alpha \geq -1$ the space $A^2_\alpha$ is actually a Hilbert space and its norm can be computed in terms of the Taylor coefficients as follows
				\[
					\| f \|_{A^2_\alpha}^2 = \sum_{n = 0}^\infty \dfrac{|a_n|^2}{c_{\alpha + 2} (n)},
				\]
				where $c_{\beta} (n) := \binom{n + \beta - 1}{n} = \frac{\Gamma (n + \beta)}{\Gamma(\beta) n!}$. Naturally, here $\Gamma$ denotes the Euler Gamma function,
				\[
					\Gamma(x) := \int_0^\infty t^{x-1} e^{-t} \: dt, \quad x > 0.
				\]
			This special function and Euler's Beta function,
				\[
					B(x,y) := \int_0^1 t^{x-1} (1 - t)^{y - 1} \: dt = \dfrac{\Gamma(x) \Gamma(y)}{\Gamma(x + y)}, \quad x, y > 0,
				\]
				will be used throughout this paper.
			
			The reproducing kernel of $A^2_\alpha$,
				\[
					K_\zeta (z) = \dfrac{1}{(1 - \overline{\zeta}z)^{\alpha + 2}}, \quad \zeta \in \D,
				\]
				is usually called the \emph{Bergman kernel} (or \emph{Szeg\H{o} kernel} in the particular case $\alpha = - 1$).
		
		\subsection{Statement of the problem and motivation}
		
			The starting point of this work is the so called \emph{Carleman's inequality}, 
				\begin{equation} \label{Carleman}
					\| f \|_{A^{2p}} \leq \| f \|_{H^p}.
				\end{equation}
				That is, the inclusion of $H^p$ in $A^{2p}$ is contractive. Furthermore, equality is possible if and only if there exists $c \in \C$ and $\zeta \in \D$ such that $f(z) = c (1 - \overline{\zeta} z)^{-\frac{2}{p}}$.
					
			 	Regarding the history of this problem, G.H. Hardy and J.E. Littlewood proved the inclusion of $H^p$ in a range of spaces that today we name as mixed norm spaces (for example, see \cite[Theorem 31]{MR1545260} or \cite[Theorem 5.11]{MR0268655}), including $A^{q}_{\frac{q}{p}-2}$ for $p < q$. However, it seems that T. Carleman \cite{MR1544458} was the first author who deduced \eqref{Carleman}, although his proof was not complete since Riesz's factorization had not been published yet. See \cite{MR577467}, \cite[Theorem 19.9]{MR743423} or \cite{MR1984405}, for example, for later references.
					
			In 1987, J. Burbea \cite{MR882113} extended Carleman's inequality to a larger class of standard weighted Bergman spaces. More specifically,
					\begin{equation}\label{Burbea}
						\| f \|_{A^{kp}_{k-2}} \leq \| f \|_{H^p},
					\end{equation}
					if $k$ is a positive integer and the unique extremal functions are precisely the same extremal functions as in Carleman's inequality.
						
			Heavily inspired by an application of such inequalities to Hardy spaces of Dirichlet series found by H. Helson \cite{MR2263964}, O.F. Brevig, J. Ortega-Cerdà, K. Seip and J. Zhao conjectured that \eqref{Burbea} should hold for any $k > 1$, not necessarily an integer. In fact, it seems that this question has been of interest to experts since even before (\emph{e.g.}, M. Pavlovi\'{c} collected it as an open question in his monograph of 2014 \cite[Problem 2.1]{MR3154590}), but it was in \cite{MR3858278} where an overwhelming evidence in favour of this conjecture was shown.
					
			After a suitable normalization using Riesz's factorization, their conjecture can be stated as follows.
					
			\begin{conj}[Brevig, Ortega-Cerdà, Seip, Zhao \cite{MR3858278}]\label{ConjectureBOSZ}
				If $p > 2$, then
					\[
						\| f \|_{A^p_{\frac{p}{2} - 2}} \leq \| f \|_{H^2}, \quad \forall f \in H^2,
					\]
					and equality is attained if and only if $f$ is a constant multiple of the Szeg\H{o} kernel.
			\end{conj}
					
			\begin{remark}
				The original statement of Conjecture \ref{ConjectureBOSZ} in \cite{MR3858278} is
					\begin{equation} \label{OriginalBOSZ}
						\sqrt{\sum_{n = 0}^\infty \dfrac{|a_n|^2}{c_{2/q}(n)}} = \| f \|_{A^2_{\frac{2}{q}-2}} \leq \| f \|_{H^q},
					\end{equation}
					if $\{ a_n \}_{n \geq 0}$ are the Taylor coefficients of $f$ and $0 < q \leq 2$.
							
				Recently, A. Kulikov \cite{MR4275085} proved that there exists $\delta = \delta(q) > 0$ such that
					\[
						\sqrt{|a_0|^2 + \dfrac{q}{2} |a_1|^2 + \delta \sum_{n = 2}^\infty \dfrac{|a_n|^2}{(n+1)^{\frac{2}{q} - 1}} }\leq \| f \|_{H^q}.
					\]
						
				This implies that \eqref{OriginalBOSZ} holds, at least, for the first two coefficients.
			\end{remark}
					
			With a view to applications in Mathematical Physics, E. Lieb and J.P. Solovej stated that a similar question to Conjecture \ref{ConjectureBOSZ} should exist for any $A^2_\alpha$ (see \cite{Lieb_Solovej_2021}, for example).
					
			\begin{conj}[Lieb, Solovej]\label{ConjLiebSolovej}
				Let $\alpha > -1$. If $p > 2$, then
					\[
						\| f \|_{A^p_{\frac{p}{2}(\alpha + 2) - 2}} \leq \| f \|_{A^2_\alpha}, \quad \forall f \in A^2_\alpha.
					\]
					In addition, equality is possible if and only if $f$ is a constant multiple of the Bergman kernel.
			\end{conj}
					
			Recalling that $H^2$ can be understood as $A^2_{-1}$, we are going to provide a unified strategy in order to study Conjectures \ref{ConjectureBOSZ} and \ref{ConjLiebSolovej}. Indeed, \eqref{LimApalpha} yields that the inclusion of $H^2$ in $A^p_{\frac{p}{2}-2}$ is contractive if Conjecture \ref{ConjLiebSolovej} is true for every $\alpha > -1$. Although the possible converse of this statement is not straightforward, there appears to be a strong feeling among the experts that a possible solution of Conjeture \ref{ConjectureBOSZ} could lead to a solution of Conjecture \ref{ConjLiebSolovej} (and, of course, vice versa).
			
			It should be noted that contractive inequalities for inclusions between weighted Bergman spaces hold in a number of other cases which are easier to analyse. This question has been recently addressed in \cite{llinares2021contractive}.
		
		\subsection{Main results}
		
			Unfortunately we are not able to prove the conjectures, but we can deduce some partial advances. More specifically, the main results of this paper can be listed as follows:
				\begin{itemize}
					\item If $\Lambda \subset \N \cup \{ 0 \}$ is a finite set, then
						\[
							\sup_{\| f \|_{A^2_\alpha} \leq 1} \left \{ \| f \|_{A^p_{\frac{p}{2}(\alpha + 2) - 2}} \right \} = \sup \left \{ \| f \|_{A^p_{\frac{p}{2}(\alpha + 2) - 2}} \: : \: \| f \|_{A^2_\alpha} \leq 1, f^{(j)}(0) = 0, j \in \Lambda \right \}.
						\]
						
					\item Using the notation
						\[
							\begin{array}{c}
								F_p(f)  :=  |f(0)|^p + \frac{p}{2} \int_\D |f'(z)|^2 |f(z)|^{p-2} (1 - |z|^2)^{\frac{p(\alpha + 2)}{2}} \int_0^1 \frac{s^{\alpha + 1}}{1 - (1 - |z|^2)s} \: ds \, dA(z), \\
								G_p(f) := \frac{p}{2(\alpha + 2)} \int_\D |f'(z)|^2 |f(z)|^{p-2} (1 - |z|^2)^{\frac{p(\alpha + 2)}{2}} \: dA(z), \nonumber
							\end{array}
						\]
						the following quantities
						\[
							\begin{array}{c}
								\sup \left \{ \| f \|_{A^p_{\frac{p}{2}(\alpha + 2) - 2}} \: : \: \| f \|_{A^2_\alpha} \leq 1 \right \}, \\
									\displaystyle{\sup \left \{ F_p^\frac{1}{p} (f) \: : \: \| f \|_{A^2_\alpha} \leq 1  \right \}}, \\
									\displaystyle{\sup \left \{ G_p^\frac{1}{p} (f) \: : \: \| f \|_{A^2_\alpha} \leq 1, f^{(j)}(0) = 0, 0 \leq j \leq n\right \}},
							\end{array}
						\]
						are equal for any $n \geq 0$.
					
					\item The inclusion $A^2_\alpha \subset A^p_{\frac{p}{2}(\alpha + 2) - 2}$ is contractive if and only if for every $f$ such that $\| f \|_{A^2_\alpha} = 1$ there exists $g$ satisfying that $\| g \|_{A^2_\alpha} = 1$, $F_p (f) \leq F_p (g)$ and
						\[
							|g(0)| = \sup \left \{ |g(z)| (1 - |z|^2)^\frac{\alpha + 2}{2} \: : \: z \in \D \right \}.
						\]
					
					\item If $2(k - 1) < p \leq 2k$ for some integer $k \geq 2$, then
						\[
							\| f \|_{A^p_{\frac{p}{2}(\alpha + 2) - 2}} \leq \left( \dfrac{p}{2} \right)^{\frac{1}{p}} \dfrac{1}{(k-1)^{\frac{k}{p}-\frac{1}{2}}k^{\frac{1}{2}-\frac{k-1}{p}}} \| f \|_{A^2_\alpha}.
						\]
				\end{itemize}
			
	\section{Vanishing functions and reformulation of the problem}
		
		First, we are going to show that the norm of the inclusion is preserved if we restrict our attention to the subspace of $A^2_\alpha$ functions whose Taylor coefficients are zero on a prescribed finite set.
		
		\begin{thm} \label{Vanishing}
			Let $\Lambda \subset \N \cup \{ 0 \}$ a finite set. Then, for every $\alpha \geq -1$ and $p > 2$, the following identity holds
				\[
					\sup_{\| f \|_{A^2_\alpha} \leq 1} \left \{ \| f \|_{A^p_{\frac{p}{2}(\alpha + 2) - 2}} \right \} = \sup \left \{ \| f \|_{A^p_{\frac{p}{2}(\alpha + 2) - 2}} \: : \: \| f \|_{A^2_\alpha} \leq 1, f^{(j)}(0) = 0 \mbox{ for all } j \in \Lambda \right \}.
				\]
		\end{thm}
				
		\begin{proof}
			Let $M_\Lambda$ be the restricted supremum on the right, and take $f \in A^2_\alpha$ such that $\| f \|_{A^2_\alpha} \leq 1$. Consider $\{ r_n \}_{n \geq 1} \subset (0, 1)$ any sequence convergent to $1$.
					
			We recall that the operators
				\begin{equation} \label{IsometriaA2alpha}
					T_a (f) (z) := \big( \varphi_a' (z) \big)^\frac{\alpha + 2}{2} f \big ( \varphi_a (z) \big), \quad a \in \D,
				\end{equation}
				are isometries in both $A^2_\alpha$ and $A^p_{\frac{p}{2}(\alpha + 2) - 2}$ and  then, using triangle inequality,
				\[
					\| f \|_{A^p_{\frac{p}{2}(\alpha + 2) - 2}} = \| f_n \|_{A^p_{\frac{p}{2}(\alpha + 2) - 2}} \leq \sum_{j \in \Lambda} \frac{|f_n^{(j)}(0)|}{j!} \| z^j \|_{A^p_{\frac{p}{2}(\alpha + 2)-2}} + M_\Lambda.
				\]
					
			Since $\Lambda$ is finite, it is enough to prove that $f_n^{(j)}(0) \rightarrow 0$ when $n \rightarrow \infty$ in order to see that $M_\Lambda$ is an admissible bound for the unrestricted supremum.
					
			Using Cauchy's integral formula, we have
				\begin{eqnarray*}
					\dfrac{|f_n^{(j)} (0)|}{j!} & = & \left | \dfrac{1}{2\pi i} \int_{|z| = \frac{1}{2}} \dfrac{f_n(z)}{z^{j+1}} dz \right | \\
						& \leq & \dfrac{2^{j-1}}{\pi} \left( \dfrac{4}{3} \right)^{\frac{\alpha + 2}{2}} \int_0^{2\pi} \big ( 1 - \big | \varphi_{r_n} (2^{-1}e^{it}) \big|^2 \big)^\frac{\alpha + 2}{2} \big | f \big ( \varphi_{r_n} (2^{-1}e^{it}) \big ) \big| \: dt.
				\end{eqnarray*}
						
			The function $f$ belongs to $A^2_\alpha$, and consequently we know that
				\[
					\lim_{|z| \rightarrow 1^-} \left \{ |f(z)| (1 - |z|^2)^\frac{\alpha + 2}{2} \right \} = 0.
				\]
					
			In addition, it is an elementary computation to see that $\left [\frac{r_n - 2^{-1}}{1 - 2^{-1}r_n}, \frac{r_n + 2^{-1}}{1 + 2^{-1}r_n} \right]$ is a diameter of the circle $\varphi_{r_n} (2^{-1} \T)$. Then, it is immediate that $f_n^{(j)}(0)$ must converge to 0.
		\end{proof}
				
		A frequent choice for $\Lambda$ will be the discrete interval $\{0, \ldots, n\}$, $n \geq 0$. In other words, we can focus on $A^2_\alpha$ functions with a zero of order arbitrary large at the origin. If $\alpha = -1$ (that is, in the $H^2$ setting) this strategy is completely against the usual procedure of cancelling zeros using Riesz's factorization. In other words, we are going to use non-standard techniques in order to deduce an equivalent formulation of the conjecture.
				
		The other key idea of this paper is that we will not use the standard expression for $\| \cdot \|_{A^p_{\frac{p}{2}(\alpha + 2)-2}}$. Instead, we are going to employ an alternative expression for this norm deduced from the \emph{Hardy-Stein identity}, property that for the sake of completeness we collect under these lines.
		
		\begin{lemalpha}[Hardy-Stein identity] \label{Hardy-Stein}
			For any $f \in \Hol$
			\[
				\dfrac{d}{dr} M_p^p(r, f) = \dfrac{p^2}{2r} \int_{r\D} |f'(z)|^2 |f(z)|^{p-2} \: dA(z), \quad \forall r \in (0,1).
			\]
		\end{lemalpha}
			
		\begin{proof}
			The proof can be found in several monographs like \cite[Section~2.6]{MR3154590} or \cite[p.~174]{MR1217706}.
		\end{proof}
		
		Indeed, using integration by parts, the Hardy-Stein identity, Fubini's theorem and an appropriate change of variables, we get that the $A^p_{\frac{p}{2}(\alpha + 2)-2}$ norm of $f \in A^2_\alpha$ can be computed as follows.
				
		\begin{eqnarray*}
			\| f \|_{A^p_{\frac{p(\alpha + 2)}{2} - 2}}^p & = & |f(0)|^p + \int_0^1 (1 - r^2)^{\frac{p(\alpha + 2)}{2} - 1} \dfrac{d}{dr} M_p^p (r, f) \: dr \\
				& = & |f(0)|^p + \dfrac{p^2}{2} \int_0^1 \dfrac{(1 - r^2)^{\frac{p(\alpha + 2)}{2} - 1}}{r} \int_{r \D} |f'(z)|^2 |f(z)|^{p - 2} \: dA(z) \, dr \\
				& = & |f(0)|^p + \dfrac{p^2}{4} \int_\D |f'(z)|^2 |f(z)|^{p-2} (1 - |z|^2)^{\frac{p(\alpha + 2)}{2}} \int_0^1 \dfrac{s^{\frac{p(\alpha + 2)}{2} - 1}}{1 - (1 - |z|^2)s} \: ds \, dA(z).
		\end{eqnarray*}
				
		This expression of $\| \cdot \|_{A^p_{\frac{p}{2}(\alpha + 2) - 2}}$ and Theorem \ref{Vanishing} are enough to prove the next equivalence. In addition, we are going to need the following standard inequality due to Chebyshev.
		
		\begin{lemalpha}[Chebyshev's inequality] \label{Chebyshev}
			Let $a < b$ real numbers and $\mu$ a probability measure in $(a.b)$. Then, if $f,g : (a,b)  \rightarrow \R$ are two functions one increasing and the other decreasing, then
			\[
				\int_a^b f(x) g(x) \: d\mu(x) \leq \int_a^b f(x) \: d\mu(x) \int_a^b g(x) \: d \mu(x).
			\]
		\end{lemalpha}
			
		\begin{proof}
			It is a consequence of
			\[
				\int_{(a,b)^2} \big ( f(x) - f(y) \big) \big ( g(x) - g(y) \big) \: d\mu(x, y) \leq  0,
			\]
			and Fubini's theorem.
		\end{proof}
				
		\begin{thm} \label{Equivalence}
			Let $\alpha \geq -1$, $p > 2$ and $C \geq 1$. The following statements are equivalent:
				\begin{enumerate}[(a)]
					\item \label{Eq1} For any $f \in A^2_\alpha$,
						\[
							\| f \|_{A^p_{\frac{p}{2}(\alpha + 2) - 2}} \leq C \| f \|_{A^2_\alpha}.
						\]
							
					\item \label{Eq2} For every $f \in A^2_\alpha$ such that $\| f \|_{A^2_\alpha} = 1$, the estimate
						\[
							|f(0)|^p + \dfrac{p}{2} \int_\D |f'(z)|^2 |f(z)|^{p - 2} (1 - |z|^2)^{\frac{p(\alpha + 2)}{2}} \int_0^1 \dfrac{s^{\alpha + 1}}{1 - (1 - |z|^2) s} \: ds \, dA(z) \leq C^p,
						\]
						holds.
								
					\item \label{Eq3} There exists $M > 0$ such that for every $f \in A^2_\alpha$ with $\| f \|_{A^2_\alpha} = 1$ and $f^{(j)} (0) = 0$, $0 \leq j \leq M$, the estimate
						\[
							\dfrac{p}{2(\alpha + 2)} \int_\D |f'(z)|^2 |f(z)|^{p - 2} (1 - |z|^2)^{\frac{p(\alpha + 2)}{2}} \: dA(z) \leq C^p,
						\]
						holds.
				\end{enumerate}
		\end{thm}
				
		\begin{proof}
			If $z \in \D \setminus \{ 0 \}$, integration by parts yields that
				\begin{eqnarray}
					H(p, z) & := & \dfrac{p}{2} \int_0^1 \dfrac{s^{\frac{p(\alpha + 2)}{2} - 1}}{1 - (1 - |z|^2)s} \: ds \nonumber \\
						& = & \dfrac{1}{\alpha + 2} \left ( \dfrac{1}{|z|^2} - (1 - |z|^2) \int_0^1 \dfrac{s^{\frac{p(\alpha + 2)}{2}}}{\big ( 1 - (1 - |z|^2) s \big)^2} ds \right). \label{AuxFunct1}
				\end{eqnarray}
							
			In particular, $H(p, z)$ is an increasing function of $p$. Then,
				\[
					|f(0)|^p + \dfrac{p}{2} \int_\D |f'(z)|^2 |f(z)|^{p - 2} (1 - |z|^2)^{\frac{p(\alpha + 2)}{2}} \int_0^1 \dfrac{s^{\alpha + 1}}{1 - (1 - |z|^2) s} \: ds \, dA(z) \leq \| f \|_{A^p_{\frac{p}{2}(\alpha + 2) - 2}}^p,
				\]
				for every $f \in A^2_\alpha$. Therefore, \ref{Eq1} implies \ref{Eq2}.
						
			That \ref{Eq2} implies \ref{Eq3} is trivial.
				
			Finally, assume that \ref{Eq3} is true. Due to Theorem \ref{Vanishing}, it is enough to prove that for any $\varepsilon > 0$ there exists a sufficiently large $m$ such that
				\[
					\| f \|_{A^p_{\frac{p}{2}(\alpha + 2) - 2}}^p \leq C^p + \varepsilon,
				\]
				if $\| f \|_{A^2_\alpha} = 1$ and $f^{(j)}(0) = 0$ for every $0 \leq j \leq m$.
					
			Observe that \eqref{AuxFunct1} yields that $H(p, z)$ is a concave function of $p$, and therefore its graph lies below the tangent line at any point. In particular,
				\[
					H(p, z) \leq H(2, z) + (p - 2) \dfrac{\partial H}{\partial p} (2, z), \quad \forall z \in \D \setminus \{ 0 \},
				\]
				for any $p > 2$. We are going to estimate the integrals
				\begin{eqnarray*}
					I_1 & := & \int_\D |f'(z)|^2 |f(z)|^{p-2} (1 - |z|^2)^\frac{p(\alpha + 2)}{2} \left(H(2,z) - \dfrac{1}{\alpha + 2} + (p - 2) \dfrac{\partial H}{\partial p} (2, z)\right) \: dA(z), \\
					I_2 & := & \int_\D |f'(z)|^2 |f(z)|^{p-2} (1 - |z|^2)^\frac{p(\alpha + 2)}{2} \: dA(z).
				\end{eqnarray*}
						
			Using the geometric series and monotone convergence theorem, we obtain
				\[
					H(p, z) = \dfrac{p}{2} \int_0^1 \sum_{n = 0}^\infty s^{\frac{p(\alpha + 2)}{2}+n-1} (1 - |z|^2)^n \: ds = \dfrac{p}{2} \sum_{n = 0}^\infty \dfrac{(1 - |z|^2)^n}{n+\frac{p(\alpha + 2)}{2}},
				\]
				and, similarly,
				\begin{equation} \label{AuxFunct1Der}
					\dfrac{\partial H}{\partial p} (2, z) = \dfrac{1}{2} \sum_{n = 1}^\infty \dfrac{n (1 - |z|^2)^n}{(n + \alpha + 2)^2}.
				\end{equation}
					
			Using \eqref{AcotacionApalpha}, \eqref{AuxFunct1Der}, Parseval's identity and monotone convergence theorem, we have
				\begin{eqnarray*}
					I_1 & \leq & \int_\D |f'(z)|^2 (1 - |z|^2)^{\alpha + 2} \left(H(2,z) - \dfrac{1}{\alpha + 2} + (p - 2) \dfrac{\partial H}{\partial p} (2, z)\right) \: dA(z) \\
						& = & \sum_{k = m + 1}^\infty |a_k|^2 k^2 \sum_{n = 1}^\infty \dfrac{\frac{p}{2}n + \alpha + 2}{(n+\alpha + 2)^2} B(k, n + \alpha + 3),
				\end{eqnarray*}
				where $\{a_k\}_{k \geq m+1}$ is the sequence of Taylor coefficients of $f$. Then, since $\| f \|_{A^2_\alpha} = 1$,
				\[
					I_1 \leq \dfrac{1}{\Gamma (\alpha + 2)}\sup_{k \geq m + 1} \left \{ \dfrac{\Gamma (k + \alpha + 2) k}{(k - 1)!} \sum_{n = 1}^\infty \dfrac{\frac{p}{2}n + \alpha + 2}{(n + \alpha + 2)^2}  B(k, n + \alpha + 3) \right \}.
				\]
					
			Observe that the last bound for $I_1$ is uniform in $f$. We are going to check that
				\[
					\lim_{k \rightarrow \infty} \dfrac{\Gamma (k + \alpha + 2) k}{(k-1)!} \sum_{n = 1}^\infty \dfrac{\frac{p}{2}n + \alpha + 2}{(n + \alpha + 2)^2}  B(k, n + \alpha + 3) = 0.
				\]
					
			Consider the probability measure $d\sigma (x) := (\alpha + 4) (1-x)^{\alpha + 3} \: dx$ in $(0,1)$. Then, Chebyshev's inequality yields that
				\begin{eqnarray*}
					B(k, n + \alpha + 3) & = & \dfrac{1}{\alpha + 4} \int_0^1 x^{k - 1} (1-x)^{n-1} \: d \sigma (x) \\
						& \leq & \dfrac{1}{\alpha + 4} \int_0^1 x^{k - 1} \: d \sigma (x) \int_0^1 (1 - x)^{n-1} \: d \sigma (x) \\
						& = & \dfrac{\Gamma(\alpha + 5) (k-1)!}{(n + \alpha + 3)\Gamma(k + \alpha + 4)}, \quad \forall n \geq 1.
				\end{eqnarray*}
					
			Thus,
				\[
					\dfrac{\Gamma (k + \alpha + 2) k}{(k-1)!} \sum_{n = 1}^\infty \dfrac{\frac{p}{2}n + \alpha + 2}{(n + \alpha + 2)^2}  B(k, n + \alpha + 3) \lesssim \dfrac{k}{(k+\alpha+3)(k+\alpha+2)}.
				\]
					
			In other words, we have proved that for every $\varepsilon > 0$ there exists $m$ large enough such that $I_1 \leq \frac{2}{p} \varepsilon$ if $f(0) = \ldots = f^{(m)}(0) = 0$.
			
Of course, without loss of generality we can assume that $m > M$ and, since we are assuming that \ref{Eq3} is true, then $\frac{p}{2(\alpha + 2)} I_2 \leq C^p$.
						
			Summing up, we have shown that
				\[
					\| f \|_{A^p_{\frac{p}{2}(\alpha + 2) - 2}}^p \leq \dfrac{p}{2(\alpha + 2)} I_2 + \dfrac{p}{2} I_1 \leq C^p + \varepsilon,
				\]
				if $\| f \|_{A^2_\alpha} = 1$ and its first $m$ Taylor coefficients are zero. Therefore, the inclusion operator from $A^2_\alpha$ to $A^p_{\frac{p}{2}(\alpha + 2)-2}$ has norm less than or equal to $C$ if \ref{Eq3} is true.
		\end{proof}
				
		As a sample application of Theorem \ref{Equivalence} to a possible solution to the conjecture, we are going to give an alternative proof for the case $p = 2k$, which has already been solved by several authors like O.F. Brevig, J. Ortega-Cerdà, K. Seip and J. Zhao \cite{MR3858278}, E.H. Lieb and J.P. Solovej \cite{Lieb_Solovej_2021} or D. Békollè, J. Gonessa and B.F. Sehba for $\alpha = 0$ \cite{bekolle2020conjecture}.
				
		\begin{thmalpha}[Extended Carleman's inequality]\label{ExtendedCarleman}
			If $\alpha \geq -1$ and $k$ is a positive integer, then
				\[
					\| f \|_{A^{2k}_{k(\alpha+2)-2}} \leq \| f \|_{A^2_\alpha}.
				\]
		\end{thmalpha}
				
		\begin{proof}
			We will provide a proof by induction. If $k = 1$ the result is obvious (the two spaces are the same). 
					
			Assume that the inclusion $A^2_\alpha \subset A^{2k}_{k(\alpha + 2) - 2}$ is contractive for some $k$, and take $f$ with $\| f \|_{A^2_\alpha} = 1$. We are going to show that
				\[
					I_{k + 1} := \dfrac{k + 1}{\alpha + 2} \int_\D |f'(z)|^2 |f(z)|^{2k} (1-|z|^2)^{(k+ 1)(\alpha + 2)} \: dA(z),
				\]
				can be bounded by $1$.
					
			Let $g = f^{k + 1}$. If $\{a_n\}_{n \geq 0}$, $\{b_n\}_{n \geq 0}$ and $\{c_n\}_{n \geq 0}$ are the sequences of Taylor coefficients of $f$, $f^k$ and $f^{k + 1}$, respectively, then $c_n = a_0 b_n + \ldots + a_n b_0$.
						
			Therefore, using the chain rule and Parseval's identity, we have
				\begin{eqnarray*}
					I_{k + 1} & = & \dfrac{1}{(k+1)(\alpha + 2)} \int_\D |g'(z)|^2 (1 - |z|^2)^{(k+1)(\alpha + 2)} \: dA(z) \\
						& = & \dfrac{1}{(k+1)(\alpha + 2)} \sum_{n = 1}^\infty n^2 B\big (n, (k+1)( \alpha + 2) + 1 \big) |c_n|^2.
				\end{eqnarray*}
					
			On the other hand, due to Cauchy-Schwarz inequality,
				\[
					|c_n|^2 \leq \left ( \sum_{j = 0}^n c_{\alpha + 2} (j) c_{k(\alpha + 2)} (n - j) \right) \left( \sum_{j= 0}^n \dfrac{|a_j|^2}{c_{\alpha + 2} (j)} \dfrac{|b_{n-j}|^2}{c_{k(\alpha + 2)}(n-j)}\right).
				\]
						
			Using the fact that $\{c_{\beta} (n) \}_{n \geq 0}$ is the sequence of Taylor coefficients of $(1 - z)^{-\beta}$, it is clear that
				\[
					\sum_{j = 0}^n c_{\alpha + 2} (j) c_{k(\alpha + 2)} (n - j) = c_{(k + 1)(\alpha + 2)} (n), \quad \forall n \geq 0,
				\]
				and therefore
				\begin{equation}\label{BetaBinom}
					\dfrac{n^2 B \big ( n, (k + 1)(\alpha + 2) + 1\big)}{(k+1)(\alpha + 2)} c_{(k+1)(\alpha + 2)}(n) = \dfrac{n}{n + (k+1)(\alpha + 2)}, \quad \forall n \geq 1.
				\end{equation}
						
			Summarizing,
				\begin{eqnarray*}
					I_{k + 1} & \leq & \sum_{n = 1}^\infty \dfrac{n}{n + (k+1)(\alpha + 2)}\sum_{j = 0}^n \dfrac{|a_j|^2}{c_{\alpha + 2}(j)} \dfrac{|b_{n-j}|^2}{c_{k(\alpha + 2)} (n - j)} \\
						& \leq & \| f \|_{A^2_\alpha}^2 \| f^k \|_{A^2_{k(\alpha + 2) - 2}}^2 \\
						& = & \| f \|_{A^{2k}_{k(\alpha + 2) - 2}}^{2k},
				\end{eqnarray*}
				and then the inclusion of $A^2_\alpha$ in $A^{2(k+1)}_{(k + 1)(\alpha + 2)-2}$ is contractive due to the induction hypothesis.
		\end{proof}				
				
		Finally, we are going to derive a new necessary and sufficient condition for the contractivity.
				
		\begin{lem} \label{AuxLemmaEquiv}
			Let $\alpha \geq -1$ and $p > 2$. If $f$ is an analytic function such that $\| f \|_{A^2_\alpha} = 1$ and $|f(0)|~=~\displaystyle{\sup_{z \in \D} \left \{ |f(z)|(1 - |z|^2)^{\frac{\alpha + 2}{2}} \right \}}$, then $F_p(f) \leq 1$. Furthermore, equality is attained if and only if $f$ is constant.
		\end{lem}
					
		\begin{proof}
			By the conditions on $f$, we have $F_p (f) \leq |f(0)|^{p - 2} \left( 1 + \frac{p}{2} (1 - |f(0)|^2)\right)$.
					
			If $p > 2$, a direct computation shows that the function $x^{p - 2} \left( x^2 + \frac{p}{2} (1 - x^2) \right)$ is a strictly increasing function of $x$ in $[0, 1]$. In other words, $F_f (p) \leq 1$.
					
			On the other hand, equality is possible only if $|f(0)| = 1 = \| f \|_{A^2_\alpha}$, and thus $f$ must be constant.
		\end{proof}
				
		\begin{thm}
			Let $p > 2$ and $\alpha \geq -1$. The following statements are equivalent:
				\begin{enumerate}[(a)]
					\item The inclusion $A^2_\alpha \subset A^p_{\frac{p}{2}(\alpha + 2) - 2}$ is contractive.
							
					\item For any $f$ such that $\| f \|_{A^2_\alpha} = 1$, there exists $g$ with $\| g \|_{A^2_\alpha} = 1$ and $|g(0)|~=~\displaystyle{\sup_{z \in \D} \left \{ |g(z)| (1 - |z|^2)^{\frac{\alpha + 2}{2}} \right \}}$ such that
						\[
							F_p (f) \leq F_p(g).
						\]
				\end{enumerate}
		\end{thm}
				
		\begin{proof}
			If the inclusion is contractive, then Theorem \ref{Equivalence} yields that
				\[
					F_p (f) \leq 1, \quad \mbox{ if } \quad \| f \|_{A^2_\alpha} = 1,
				\]
				an so the constant function $1$ is a possible choice of $g$.
						
			On the other hand, if we can find such a function $g$ for each normalized $f \in A^2_\alpha$, then Lemma \ref{AuxLemmaEquiv} implies that
				\[
					F_p (f) \leq F_p (g) \leq 1,
				\]
				and Theorem \ref{Equivalence} yields that the inclusion of $A^2_\alpha$ in $A^p_{\frac{p}{2}(\alpha + 2) - 2}$ is contractive.
		\end{proof}
				
		Unfortunately, the quantity $F_p(f)$ is not invariant under the isometries $T_a$ introduced in \eqref{IsometriaA2alpha}. A straightforward counterexample is given by the normalized reproducing kernels. That is, if $p > 2$, $\zeta \in \D \setminus \{ 0 \}$ and $\tilde{K}_\zeta = \frac{K_\zeta}{\| K_\zeta \|_{A^2_\alpha}}$, we see that
			\begin{eqnarray*}
				F_p ( \tilde{K}_\zeta ) < \| \tilde{K}_\zeta \|_{A^p_{\frac{p}{2}(\alpha + 2) - 2}} = 1 = F_p \big ( T_\zeta (\tilde{K}_\zeta) \big),
			\end{eqnarray*}
			since $T_\zeta (\tilde{K}_\zeta)$ is constant.
				
	\section{Explicit bounds}
				
		Finally, we will find explicit values for $C$ in order to get
			\[
				\| f \|_{A^p_{\frac{p}{2}(\alpha + 2) - 2}} \leq C \| f \|_{A^2_\alpha}, \quad \forall f \in A^2_\alpha,
			\]
			if $p$ is not an even integer.
				
		\begin{cor}\label{UniformBound}
			If $\alpha \geq -1$ and $2(k - 1) < p \leq 2k$ for some integer $k \geq 2$, then
				\[
					\| f \|_{A^p_{\frac{p}{2}(\alpha + 2) - 2}} \leq \left( \dfrac{p}{2} \right)^{\frac{1}{p}} \dfrac{1}{(k-1)^{\frac{k}{p}-\frac{1}{2}}k^{\frac{1}{2}-\frac{k-1}{p}}} \| f \|_{A^2_\alpha}.
				\]
		\end{cor}
				
		\begin{proof}
			Take $f \in A^2_\alpha$ such that $\| f \|_{A^2_\alpha} = 1$ and $f(0) = \ldots = f^{(m)}(0) = 0$. We are going to prove that
				\[
					I_p := \dfrac{p}{2(\alpha + 2)} \int_\D |f'(z)|^2 |f(z)|^{p-2} (1 - |z|^2)^{\frac{p(\alpha + 2)}{2}} \: dA(z) \leq \dfrac{p}{2} \dfrac{1}{(k - 1)^{k-\frac{p}{2}}k^{\frac{p}{2}-k+1}} + o(1),
				\]
				if $m \rightarrow \infty$.
				
			Note that, because of \eqref{BetaBinom} and Theorem \ref{ExtendedCarleman},
				\begin{eqnarray*}
					\int_\D |f'(z)|^2 |f(z)|^{2(k-2)} (1 - |z|^2)^{(k-1)(\alpha + 2)} \: dA(z) & \leq & \dfrac{\alpha + 2}{k - 1} + o(1), \\
					\int_\D |f'(z)|^2 |f(z)|^{2(k-1)} (1 - |z|^2)^{k(\alpha + 2)} \: dA(z) & \leq & \dfrac{\alpha + 2}{k} + o(1),
				\end{eqnarray*}
				if $m \rightarrow \infty$ uniformly in $f$.
						
			Since $p \in \big( 2(k-1), 2k \big ]$, then $p^{-1} = \theta \big( 2(k-1) \big)^{-1} + (1 - \theta) (2k)^{-1}$ with
				\[
					\theta = (k-1) \left( \dfrac{2k}{p} - 1 \right) \quad \mbox{ and } \quad 1 - \theta = k \left(1 - \dfrac{2(k-1)}{p} \right),
				\]
				and therefore Hölder's inequality yields
				\[
					I_p \leq \dfrac{p}{2} \dfrac{1}{(k-1)^{k-\frac{p}{2}}k^{\frac{p}{2}-k+1}} + o(1),
				\]
				when $m \rightarrow \infty$.
		\end{proof}
				
		For the sake of clarity, $C(p)$ will denote the bound deduced in Corollary \ref{UniformBound}.
			\[
				C(p) := \left( \dfrac{p}{2} \right)^{\frac{1}{p}} \dfrac{1}{(k-1)^{\frac{k}{p}-\frac{1}{2}}k^{\frac{1}{2}-\frac{k-1}{p}}}, \quad 2(k-1) < p \leq 2k.
			\]
					
		Because of $C(2k) = 1$ and that $C(p) \geq 1$ for all $p$, a straightforward computation shows that the maximum of $C$ in $\big (2(k-1), 2k \big ]$ is attained at $p = 2e \frac{(k-1)^k}{k^{k-1}}$, and therefore
			\[
				C(p) \leq \sqrt{\dfrac{k - 1}{k}} e^{\frac{k^{k-1}}{2e(k-1)^k}}, \quad \forall p \in \big ( 2(k - 1), 2k \big].
			\]
					
		\begin{lem} \label{MonoAuxSeq}
			The function
				\[
					\sqrt{\dfrac{x-1}{x}} e^{\frac{x^{x-1}}{2e(x-1)^x}}, \quad x > 1,
				\]
				is decreasing.
		\end{lem}
					
		\begin{proof}
			We are going to show that
				\[
					I(x) := \log \left( 1 - \dfrac{1}{x} \right) + \dfrac{1}{e} \left(1 + \dfrac{1}{x - 1} \right)^{x-1} \dfrac{1}{x-1},
				\]
				is decreasing in $(1, \infty)$.
						
			It is a direct computation to show that
				\[
					I' (x) = \dfrac{1}{x(x-1)} \left [1 + \dfrac{1}{e} \left( 1 + \dfrac{1}{x-1}\right)^{x-1} \left( x \log \left(1 + \dfrac{1}{x-1} \right) - 2 - \dfrac{1}{x-1} \right) \right],
				\]
				and consequently it is enough to see that
				\[
					J (x) := \left(1 + \dfrac{1}{x-1}\right)^{x-1} \left( x \log \left(1 + \dfrac{1}{x-1} \right) - 2 - \dfrac{1}{x-1} \right), \quad x > 1,
				\]
				is increasing.
							
			Similarly,
				\[
					J'(x) = \left( 1 + \dfrac{1}{x-1} \right)^{x - 1} \left( x \log^2 \left(1 + \dfrac{1}{x-1} \right) - \dfrac{2x-1}{x-1} \log \left(1+\dfrac{1}{x-1}\right) + \dfrac{x^2-x+1}{x(x-1)^2} \right).
				\]
						
			Observe that the discriminant of the quadratic trinomial
				\[
					P_x (y) := x^2 (x - 1)^2 y^2 - (2x-1) x (x-1) y + x^2 - x + 1,
				\]
				is negative. Thus, $P_x(y) \geq 0$ for all $y$ and then $J$ is increasing.
		\end{proof}
					
		From Lemma \ref{MonoAuxSeq} we deduce the uniform bound
			\[
				\| f \|_{A^p_{\frac{p}{2}(\alpha + 2)-2}} \leq \dfrac{\sqrt[e]{e}}{\sqrt{2}} \| f \|_{A^2_\alpha} = 1.021\ldots \| f \|_{A^2_\alpha}, \quad \forall f \in A^2_\alpha,
			\]
			for all $p > 2$.
						
		The best known constant so far was due to O.F. Brevig, J. Ortega-Cerdà, K. Seip and J. Zhao \cite{MR3858278}. They proved that
			\[
				\| f \|_{A^p_{\frac{p}{2}-2}} \leq \sqrt{\dfrac{2}{e\log 2}} \| f \|_{H^2} = 1.03\ldots \| f \|_{H^2}, \quad \forall f \in H^2,
			\]
			if $p \geq 4$. Thus, Corollary \ref{UniformBound} represents a slight improvement (both in the value of the uniform bound and in the range of admissible $p$ and $\alpha$) with respect to the previously published works.
					
		\begin{figure}[h!]
			\centering
			\includegraphics[scale=0.64]{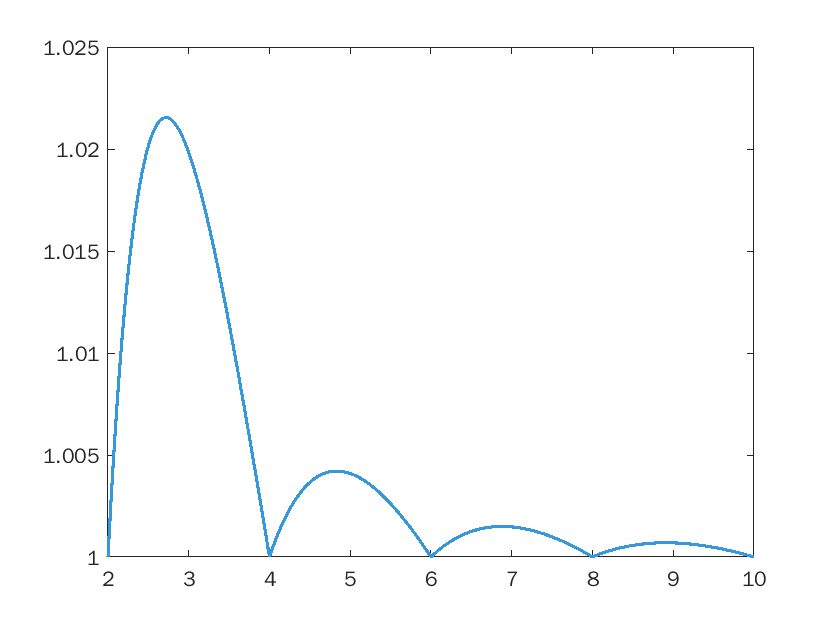}
			\caption{Plot of the function $C(p)$ in $(2,10)$ using \emph{Matlab}.}
		\end{figure}
			
		Another consequence of Corollary \ref{UniformBound} is that $H^2$ is the limit of $A^p_{\frac{p}{2}-2}$ in the sense of \eqref{LimApalpha}.
				
		\begin{cor}
			If $f \in H^2$,
				\[
					\lim_{p \rightarrow 2^+} \| f \|_{A^p_{\frac{p}{2}-2}} = \| f \|_{H^2}.
				\]
		\end{cor}
			
		\begin{proof}
			Without loss of generality, assume $2 < p < 4$. Due to the monotonicity of $M_p(r,f)$, we can apply Hölder's inequality and Chebyshev's inequality (the version for two functions with the same monotonicity) to deduce
				\begin{eqnarray*}
					\| f \|_{A^p_{\frac{p}{2}-2}}^p & = & \left( \dfrac{p}{2} - 1 \right) \int_0^1 2r (1 - r^2)^{\frac{p}{2}-2} M_p^p(r, f) \: dr \\
						& \geq & \left( \dfrac{p}{2} - 1 \right) \int_0^1 2r (1 - r^2)^{\frac{p}{2}-2} M_2^p(r, f) \: dr \\
						& \geq & \left( \dfrac{p}{2} - 1 \right) \int_0^1 2r (1 - r^2)^{\frac{p}{2}-2} M_2^2(r, f) \: dr \int_0^1 2r M_2^{p - 2} (r, f) \: dr.
				\end{eqnarray*}
					
			That is, Corollary \ref{UniformBound} (for $\alpha = -1$) implies
				\[
					\| f \|_{A^2_{\frac{p}{2}-2}}^2 \int_0^1 2r M_2^{p - 2} (r, f) \: dr \leq \| f \|_{A^p_{\frac{p}{2}-2}}^p \leq \dfrac{p}{2^\frac{p}{2}} \| f \|_{H^2}^p,
				\]
				and the convergence follows from \eqref{LimApalpha} and dominated convergence theorem.
		\end{proof}
			
		We end this paper with a final observation. Because of
			\[
				\lim_{p \rightarrow 2^{+}} \| f \|_{A^p_{\frac{p}{2}(\alpha + 2) - 2}} = \| f \|_{A^2_\alpha},
			\]
			another open (and harder) problem is the following: Given $f \in A^2_\alpha$, is it true that its $A^p_{\frac{p}{2}(\alpha + 2) - 2}$ norm is a monotonic function of $p$?
				
		The best result so far on that question is due to F. Bayart, O.F. Brevig, A. Haimi, J. Ortega-Cerdà and K.-M. Perfekt \cite{MR3885157}.
				
		\begin{thmalpha}[Bayart, Brevig, Haimi, Ortega-Cerdà, Perfekt \cite{MR3885157}] Let $0 < p$ and $\beta > \frac{\sqrt{17}-7}{4}$. Then
			\[
				\| f \|_{A^{\frac{p(\beta + 3)}{\beta + 2}}_{\beta + 1}} \leq \| f \|_{A^p_\beta}, \quad \forall f \in A^p_\beta
			\]
			and equality holds if and only if
			\[
				f(z) = \dfrac{c}{(1 - \overline{\zeta} z)^\frac{2(\beta + 2)}{p}},
			\]
			for some $c \in \C$ and $\zeta \in \D$.
		\end{thmalpha}
				
		Choosing $\beta = \frac{p}{2}(\alpha + 2) - 2$, we have
			\[
				\| f \|_{A^{p + \frac{2}{\alpha + 2}}_{\frac{p + 2/(\alpha + 2)}{2}(\alpha + 2)-2}} \leq \| f \|_{A^p_{\frac{p}{2}(\alpha + 2)-2}}.
			\]
			
		In other words, the monotonicity is true whenever we take ``steps of length $\frac{2}{\alpha + 2}$''.
		
	\section*{Acknowledgments} 
		
		The author wants to thank Professors D. Vukoti\'{c} and K. Seip for the interesting discussions about the problem treated in this paper and also to O.F. Brevig and K.-M. Perfekt for their comments. In general, the author also wants to thank the members of Department of Mathematical Sciences at NTNU for their hospitality during his stay there from August to November of 2021. This stay was funded by grant EST21/00461 from MU (Spain). The author is partially supported by grant PID2019-106870GB-I00 from MICINN and by MU Fellowship, reference number FPU17/00040.
		
	\nocite{MR0268655}
	\bibliographystyle{}
	\bibliography{Bibliografia}

\end{document}